\documentclass{amsart}
\usepackage[utf8]{inputenc}
\usepackage{amsmath}
\usepackage{amssymb}
\usepackage{amsthm}
\usepackage{enumitem}
\usepackage{todonotes}
\usepackage{hyperref}
\usepackage{url}
\usepackage{xcolor}

\usepackage{comment}

\setlist[enumerate]{label=\rm{(\arabic*)}, ref=(\arabic*)}

\DeclareMathOperator{\St}{St}
\DeclareMathOperator{\Aut}{Aut}
\DeclareMathOperator{\rist}{Rist}

\newcommand{\N}{\mathbb{N}}

\newcommand*{\level}[1]{\mathcal L_{#1}}

\newtheorem{thm}{Theorem}[section]
\newtheorem{lemma}[thm]{Lemma}
\newtheorem{prop}[thm]{Proposition}

\newtheorem{thmx}{Theorem}[section]

\newtheorem{corx}[thmx]{Corollary}

\newtheorem*{MainThm}{Theorem \ref{thm:MainThm}}

\theoremstyle{definition}

\newtheorem{defn}[thm]{Definition}

\title{On quasi-$2$-transitive actions of branch groups}
\author{Dominik Francoeur}

\begin{document}

\begin{abstract}
An action of a group $G$ on a set $X$ is said to be quasi-$n$-transitive if the diagonal action of $G$ on $X^n$ has only finitely many orbits. We show that branch groups, a special class of groups of automorphisms of rooted trees, cannot act quasi-2-transitively on infinite sets.
\end{abstract}

\maketitle

\section{Introduction}

A group is often best understood through its actions. For this reason, it is natural to ask what kind of actions a given group may admit. This question is of course very broad and can be taken in various interesting directions, depending on the kind of actions and the kind of properties one wishes to study.
In this article, we will be interested in what can be said of any action of an infinite group on a set, with no additional structure imposed on the set.

In this setting, a very natural question to ask is how transitive such an action can be. Recall that an action of a group $G$ on a set $X$ is said to be \emph{$n$-transitive} if any $n$-tuple of distinct points of $X$ can be sent to any other by an element of $G$.
This leads the the notion of the \emph{transitivity degree} of a group $G$, which is defined as the supremum of all $n\in \N$ such that $G$ admits a faithful $n$-transitive action. The transitivity degree is well-understood for finite groups (see for instance \cite{DixonMortimer96} and the references therein) and has been studied for various families of infinite groups. For instance, it was shown by Hull and Osin to be infinite for every countable acylindrically hyperbolic group with trivial finite radical \cite{HullOsin16}, thus generalising many previous results, including results on non-abelian free groups \cite{McDonough77}, surface groups \cite{Kitroser12}, non-elementary hyperbolic groups \cite{Chaynikov12} and outer automorphisms of free groups \cite{GarionGlasner13}. On the other end of the spectrum, it is known to be $1$ for every infinite residually finite virtually solvable group (see \cite[Corollary 4.6]{HullOsin16}), and it is known to be finite for every group satisfying a mixed identity, unless it contains a normal subgroup isomorphic to the group of finitary alternating permutations of an infinite set (\cite[Proposition A.1]{LeBoudecMatteBon19Pre}). Le Boudec and Matte Bon have also shown that the transitivity degree of certain groups of homeomorphisms of the circle and of automorphisms of trees have transitivity degree at most $3$ \cite{LeBoudecMatteBon19Pre}, a result that was later used by Fima, Le Maître, Moon and Stalder to show that for groups of automorphisms of a tree whose action is minimal and of general type, the transitivity degree is either at most $3$ or is infinite \cite{FimaLeMaitreMoonStalder20Pre}.

Instead of $n$-transitivity, other, slightly weaker, properties of the action of a group on a set can also be considered. For instance, for $n\in \N$, an action of a group $G$ on a set $X$ is said to be \emph{$n$-set-transitive}, or \emph{$n$-homogeneous}, if the action of $G$ is transitive on subsets of $X$ of size $n$. For finite groups, this property and its relation to $n$-transitivity are well-understood \cite{LivingstoneWagner65, Kantor72}. In the context of infinite groups, this was studied notably by Cameron, who showed that a group of permutation of an infinite set that is $n$-set-transitive for all $n\in \N$ but is not $n$-transitive for all $n\in \N$ must either preserve or reverse a linear or circular order \cite{Cameron76}.

Another possible generalisation of $n$-transitivity is what we call here \emph{quasi-$n$-transitivity}. The action of a group $G$ on a set $X$ is said to be quasi-$n$-transitive if the action of $G$ on $X^n$ has only finitely many orbits. If the action of $G$ on $X$ is quasi-$n$-transitive for all $n\in \N$, then it is said to be \emph{oligomorphic}. 
Groups with oligomorphic actions have been studied, notably in connection with model theory (see \cite{Cameron09}, for instance).
In a completely different direction, finitely generated groups with quasi-$2$-transitive actions were studied by de Cornulier \cite{Cornulier06} in relation with the finite presentation of permutational wreath products. Indeed, he showed that if $W$ is a non-trivial group and $G$ is a group with a given action on a set $X$, the permutational wreath product $W\wr_X G$ is finitely presented if and only if both $W$ and $G$ are finitely presented, the stabilisers of the action of $G$ on $X$ are finitely generated and the action of $G$ on $X$ is quasi-$2$-transitive.

It is easy to see that
\[ n\text{-transitivity} \Rightarrow n\text{-set-transitivity} \Rightarrow \text{quasi-}n\text{-transitivity}.\]
Thus, obtaining a bound on the possible values of $n$ such that a group admits a quasi-$n$-transitive action immediately yields bounds on the other two transitivity properties. In this article, we obtain precisely such a bound for a large family of groups known as \emph{branch groups}. Branch groups are a special class of groups of automorphisms of rooted trees (see Definition \ref{defn:BranchGroups}), first formally defined by Grigorchuk in 1997 in order to capture some of the essential properties of the group now known as the Grigorchuk group, which was the first example of a group of intermediate growth  \cite{Grigorchuk83}. In addition to being a rich source of groups with properties that are difficult to find elsewhere (such as the aforementioned intermediate growth, or groups that are infinite, finitely generated and periodic), branch groups also naturally appear in the classification of just-infinite groups \cite{Grigorchuk00}, and are the subject of active research.

Our main result is the following.
\begin{thmx}\label{thm:MainThm}
A branch group cannot act quasi-$2$-transitively on an infinite set.
\end{thmx}
As a direct corollary, we obtain
\begin{corx}
A branch group cannot act $2$-set transitively or $2$-transitively on an infinite set. In particular, the transitivity degree of a branch group is $1$.
\end{corx}
Note, however, that this last result on the transitivity degree of branch groups was already known by Adrien Le Boudec prior to this work, through a different and shorter argument that we present below in Proposition \ref{prop:TransitivityDegree}. It was in fact a source of motivation for the present article, since in general, the transitivity degree and quasi-transitivity degree of a group need not be equal, and can in certain cases be very different (here, by \emph{quasi-transitivity degree}, we mean the supremum of all $n$ such that the group admits a faithful quasi-$n$-transitive action). For example, Thompson's group $T$ has transitivity degree $2$ \cite[Corollary 1.3]{LeBoudecMatteBon19Pre} but infinite quasi-transitivity degree. As Theorem \ref{thm:MainThm} shows, however, this kind of phenomenon cannot occur in branch groups.

Our main result can also be seen as an extension of the study of maximal subgroups of branch groups, a topic that has attracted a lot of attention in recent years \cite{Pervova05,AlexoudasKlopschThillaisundaram16,FrancoeurGarrido18,KlopschThillaisundaram18}. Indeed, by a result of de Cornulier (\cite[Lemmma 3.2]{Cornulier06}, see also Proposition \ref{prop:FiniteBiIndexAlmostMaximal} below), if every almost maximal subgroup (i.e., contained in only finitely many subgroups) of a group is of finite index, then the group cannot act quasi-$2$-transitively on an infinite set. It is known that the condition that every almost maximal subgroup is of finite index does not hold for every branch groups \cite{Bondarenko10,FrancoeurGarrido18}, but Theorem \ref{thm:MainThm} shows that, nevertheless, the conclusion holds for all branch groups.

From the work of de Cornulier on finite presentations of permutational wreath products \cite{Cornulier06}, we also obtain the following corollary.

\begin{corx}
Let $G$ we a branch group acting on an infinite set $X$, and let $W$ be a non-trivial group. Then, the group $W\wr_XG$ is never finitely presented.
\end{corx}

The existence of a finitely presented branch group is currently an open question, but as the previous corollary shows, even if such a group existed, its permutational wreath products over infinite sets could never be finitely presented. Such permutational wreath products with a branch group over an infinite set were notably used by Bartholdi and Erschler to construct groups with prescribed intermediate growth \cite{BartholdiErschler12, BartholdiErschler14}. As our corollary shows, there is no hope of using their method to obtain finitely presented groups of prescribed intermediate growth.

The article is organised as follows. In Section \ref{sec:Preliminaries}, we collect a few facts and definitions that will be useful throughout the article. In Section \ref{sec:ProdenseSubgroups}, we prove some technical results about what we call quasi-prodense subgroups of branch groups that will be crucial in the proof. Lastly, in Section \ref{sec:Proof}, we prove Theorem \ref{thm:MainThm}.

\subsection*{Acknowledgements} The author would like to thank Adrien Le Boudec for suggesting this problem and for many useful discussions regarding this work. This work was performed within the framework of the LABEX MILYON (ANR-10-LABX-0070) of Université de Lyon, within the program "Investissements d'Avenir" (ANR-11-IDEX- 0007) operated by the French National Research Agency (ANR).

\section{Preliminaries}\label{sec:Preliminaries}

\subsection{Quasi-2-transitive actions and finite bi-index}

In this section, we record a few facts about quasi-2-transitive actions and their relation with subgroups of finite bi-index that will be useful later on.

Recall from the introduction that an action of a group $G$ on a set $X$ is said to be \emph{quasi-2-transitive} if the number of orbits of the diagonal action of $G$ on $X^2$ is finite. Of course, any action on a finite set is trivially quasi-2-transitive, so in what follows, we will only concern ourselves with actions on infinite set. As we will see, the existence of such an action is equivalent to the existence of an infinite index subgroup of \emph{finite bi-index}.

\begin{defn}
Let $G$ be a group and let $H\leq G$ be a subgroup. We say that $H$ has \emph{finite bi-index} if the set $H\backslash G / H$ of double cosets of $H$ is finite.
\end{defn}

\begin{prop}\label{prop:Infinite2TransitiveFiniteBiIndex}
A group $G$ admits a quasi-$2$-transitive action on an infinite set if and only if there exists an infinite index subgroup $H<G$ of finite bi-index.
\end{prop}
\begin{proof}
Let $H<G$ be a subgroup of infinite index and let us consider the action of $G$ on the infinite set $G/H$. We claim that this action is quasi-2-transitive if and only if $H$ is of finite bi-index in $G$. Indeed, for $g_1,g_2,g_3,g_4\in G$, the couples $(g_1H,g_2H)$ and $(g_3H,g_4H)$ are in the same orbit under the diagonal action of $G$ on $G/H\times G/H$ if and only if $(H, g_1^{-1}g_2H)$ is in the same orbit as $(H, g_3^{-1}g_4H)$, which is the case if and only if $g_3^{-1}g_4\in H g_1^{-1}g_2H$. Therefore, the number of orbits is finite if and only if the set $H\backslash G / H$ of double cosets is finite.

To conclude the proof, we simply remark that if $G$ admits a quasi-2-transitive action on an infinite set $X$, then there must exist some $x\in X$ whose orbit under the action of $G$ is infinite. Then, by letting $H=\St_G(x)$ be the stabiliser of this point, we notice that $H$ is of infinite index and that the action of $G$ on $G/H$ is also quasi-2-transitive.
\end{proof}

In light of the previous proposition, the following definition, due to de Cornulier \cite[Definition 3.9]{Cornulier06}, will be very useful.

\begin{defn}\label{defn:PropertyBF}
Let $G$ be a group. We say that $G$ has property $BF$ if every subgroup of finite bi-index has finite index.
\end{defn}

Thus, a group admits a quasi-$2$-transitive action on an infinite set if and only if it does not have property $BF$.

In \cite{Cornulier06}, de Cornulier studied various consequences of property $BF$ and different obstructions for a group to be of finite bi-index. We will make use of the following result, for which we include a proof for completeness.

\begin{prop}[{\cite[Lemma 3.2]{Cornulier06}}]\label{prop:FiniteBiIndexAlmostMaximal}
Let $G$ be a group and let $H$ be a subgroup of finite bi-index. Then, $H$ is \emph{almost maximal}, meaning that there are only finitely many subgroups of $G$ containing $H$.
\end{prop}
\begin{proof}
Let $L\leq G$ be a subgroup containing $H$. Then, for every $g\in L$, we have $HgH\subseteq L$. As $H$ is of finite bi-index in $G$, there are only finitely many, let us say $n$, different double cosets, from which we conclude that there can be at most $2^n$ different subgroups of $G$ containing $H$.
\end{proof}

We will also need the following elementary fact about the bi-index of subgroups.

\begin{prop}\label{prop:IntersectionFiniteBiIndex}
Let $G$ be a group, let $H\leq G$ be a subgroup of finite bi-index, and let $K\leq G$ be a subgroup of finite index. Then, $H\cap K$ is of finite bi-index in $K$.
\end{prop}
\begin{proof}
Let $n$ be the bi-index of $H$ in $G$ and let $m$ be the index of $H\cap K$ in $H$. We claim that the bi-index of $H\cap K$ in $K$ is at most $nm^2$. Let us choose a transversal $T=\{t_1, \dots, t_n\}$ for the double cosets of $H$ in $G$ and left and right transversals $L=\{l_1,\dots, l_m\}$ and $R=\{r_1,\dots, r_m\}$ for the cosets of $H\cap K$ in $H$. Now, let $g\in K$ be an arbitrary element. Then, there exist $i\in \{1,\dots, n\}$, $j,k\in \{1,\dots, m\}$ and $h_1,h_2\in H\cap K$ such that $g=h_1r_jt_il_kh_2$. Thus, we conclude that
\[K = \bigcup_{i,j,k} (H\cap K) r_jt_il_k (H\cap K)\]
where the union is taken over all $i,j,k$ such that $r_jt_il_k\in K$. The result follows.
\end{proof}

\subsection{Branch groups}

In this section, we recall the definition of branch groups and collect a few of their properties that will be useful in the sequel. We refer the reader to \cite{BartholdiGrigorchukSunic03} for a more detailed introduction to branch groups.

\subsubsection{Spherically homogeneous rooted trees}\label{subsubsec:RootedTrees}

Let $(A_i)_{i\in \N}$ be a sequence of finite sets, that we will call \emph{alphabets}, such that $|A_i|\geq 2$ for every $i\in \N$. For every $k\in \N_{>0}$, we define the set
\[(A_i)_{i\in \N}^{k} = A_0 \times A_{1} \times \dots \times A_{k-1}.\]
An element $v=(a_0, a_{1}, \dots, a_{k-1})$ of this set will be called a \emph{word of length $k$ in the alphabet $(A_{i})_{i\in \N}$}. By convention, we set $(A_i)_{i\in\N}^0 = \{\epsilon\}$, where $\epsilon$ is an element that we will call the \emph{empty word}. We will denote by
\[(A_i)_{i\in \N}^{*} = \sqcup_{k=0}^{\infty} (A_i)_{i\in\N}^k\]
the set of all words in the alphabet $(A_i)_{i\in \N}$. For $v\in (A_i)_{i\in \N}^*$, we will call the \emph{length of $v$}, denoted by $|v|$, the unique $k\in \N$ such that $v\in (A_i)_{i\in \N}^k$. In other words, if $|v|=k$, then there exist $a_i\in A_i$ for all $0\leq i < k$ such that $v=(a_0, a_1, \dots, a_{k-1})$.

For any given $k\in \N$, there is a natural operation of concatenation between a word in $(A_i)_{i\in \N}^k$ and a word in $(A_{i+k})_{i\in \N}^*$. Indeed, if $v=(a_0, \dots, a_{k-1})$ is a word of length $k$ in the alphabet $(A_i)_{i\in \N}$ and $w=(b_0, \dots, b_{l-1})$ is a word of length $l$ in the alphabet $(A_{i+k})_{i\in \N}$, then we can define the word $vw=(a_0,\dots, a_{k-1}, b_0, \dots, b_{l-1})$, which is a word of length $k+l$ in the alphabet $(A_i)_{i\in \N}$. By convention, for any $v\in (A_i)_{i\in \N}^*$, we have $\epsilon v = v\epsilon = v$.

Let $v,w\in (A_i)_{i\in \N}^*$ be two words such that $|v|\leq |w|$. We say that $v$ is a \emph{prefix} of $w$ if there is a word $v'\in (A_{i+|v|})_{i\in \N}^*$ such that $vv'=w$. The notion of prefix allows us to define a partial order on the set $(A_i)_{i\in \N}^*$. Indeed, we will say that $v\leq w$ if $v$ is a prefix of $w$. Notice that the empty word is a prefix of every word and is thus the least element of this partial order.

Let $T=((A_i)_{i\in \N}^*, E, \epsilon)$ be the rooted graph (that is, a graph with a distinguished vertex called the \emph{root}) whose vertex set is $(A_i)_{i\in \N}^*$, whose root is $\epsilon$ and whose edge set is
\[E=\big\{ \{v,w\}\subset (A_i)_{i\in \N}^* \bigm\vert v\leq w, |w|=|v|+1 \big\}.\]
It is easy to deduce from the definition that the graph thus constructed is a rooted tree, meaning that it is a connected and simply connected rooted graph. It also follows easily from the definition that the length of a vertex in this tree (in the sense of the length of words defined above) corresponds to the combinatorial distance between the root and this vertex. For $n\in \N$, we will call the set $(A_i)_{i\in \N}^n$ of all vertices of length $n$ the \emph{$n$\textsuperscript{th}} level of the tree, and we will often prefer to denote this set by $\level{n}$. By an abuse of notation, we will also frequently write $v\in T$ instead of $v\in (A_i)_{i\in \N}^*$ to mean that $v$ is a vertex of the rooted tree $T$.

If $v,w\in T$ are two vertices, we will say that $v$ is an \emph{ancestor} of $w$ and that $w$ is a \emph{descendant} of $v$ if $v\leq w$. In the particular case where $v\leq w$ with $|v|+1=|w|$, we will say that $v$ is the \emph{parent} of $w$ and that $w$ is a \emph{child} of $v$.

It is immediate from the construction of the rooted tree described above that two vertices on the same level always have the same degree. A rooted tree satisfying this property is said to be \emph{spherically homogeneous}.

It is easy to see that any spherically homogeneous rooted tree whose vertices have at least two children is isomorphic to a rooted tree of the form $T=((A_i)_{i\in \N}^*, E, \epsilon)$ as described above. In what follows, when we say \emph{rooted tree}, we will always mean a spherically homogeneous rooted tree of the form $T=((A_i)_{i\in \N}^*, E, \epsilon)$ for some sequence of alphabets $(A_i)_{i\in \N}$ such that $|A_i|\geq 2$ for all $i\in \N$, even though we will not generally mention this sequence of alphabets explicitly.

\subsubsection{Rooted subtrees}

Let $T=((A)_{i\in \N}^*, E, \epsilon)$ be a spherically homogeneous rooted tree, as described in the above section. For any vertex $v\in T$, we can define a rooted subtree $T_v$ of $T$ as the subgraph spanned by all the descendants of $v$, rooted at $v$. Thus, every vertex of $T_v$ is a word of the form $vw$ for some $w\in (A_{i+|v|})_{i\in \N}^*$. Let us denote by $T_{|v|} = ((A_{i+|v|})_{i\in\N}^*, E, \epsilon)$ the rooted tree corresponding to the sequence of alphabets $(A_{i+|v|})_{i\in \N}$. It is clear that we have an isomorphism of rooted trees between $T_v$ and $T_{|v|}$ obtained simply by deleting the prefix $v$ from any vertex of $T_v$. Using this, we also have a canonical isomorphism between $T_v$ and $T_w$ for any $v,w\in T$ on the same level, obtained simply by replacing the prefix $v$ with the prefix $w$. In what follows, we will frequently make use of these isomorphisms without mentioning them.

\subsubsection{Automorphisms of rooted trees}

Let $T=((A)_{i\in \N}^*, E, \epsilon)$ be a spherically homogeneous rooted tree, as described in Section \ref{subsubsec:RootedTrees}. An automorphism of $T$ is simply an automorphism of graph preserving the root. The group of all automorphisms of $T$ will be denoted by $\Aut(T)$.

Notice that an automorphism of $T$ preserves the length of the vertices and the partial order. It follows that there is a well-defined map
\[\varphi_v\colon \Aut(T) \longrightarrow \Aut(T_{|v|})\]
such that for every $f\in \Aut(T)$ and $w\in T_{|v|}$, we have
\[f(vw) = f(v)\varphi_v(f)(w).\]
Note that this map is in general not a group homomorphism. Indeed, for $g_1,g_2\in \Aut(T)$ and $v\in T$, we have
\[\varphi_v(g_1g_2) = \varphi_{g_2v}(g_1)\varphi_v(g_2).\]
From the above formula, however, it follows that it does become a homomorphism when restricted to the stabiliser of the vertex $v$. From the definition, we also see easily that if $v\in T$ and $w\in T_{|v|}$, then $\varphi_{vw} = \varphi_w \circ \varphi_v$.

For $f\in \Aut(T)$ and $v\in T$, the automorphism $\varphi_v(f)$ is often called the \emph{section of $f$ at $v$}. For this reason, we will call the map $\varphi_v$ the \emph{section map at $v$}.

\subsubsection{Subgroups of automorphisms of rooted trees}

Let $G\leq \Aut(T)$ be a group of automorphisms of the rooted tree $T$. We can define many subgroups of $G$ related to its action on $T$, which we do below.

For $v\in T$, we will denote by $\St_G(v)$ the stabiliser of $v$ in $G$. For $n\in \N$, we will denote by $\St_G(n) = \bigcap_{v\in \level{n}}\St_G(v)$ the subgroup of $G$ fixing every vertex of the $n$\textsuperscript{th} level of $T$. We will call this subgroup the \emph{stabiliser of the $n$\textsuperscript{th} level}, or the \emph{$n$\textsuperscript{th}-level stabiliser}.

For $v\in T$, the \emph{rigid stabiliser of $v$ in $G$}, denoted by $\rist_G(v)$, is the subgroup
\[\rist_G(v) = \left\{g\in G \mid gw=w \quad \forall w\notin T_v \right\}.\]
In other words, the rigid stabiliser of a vertex $v$ is the subgroup of all elements of $G$ whose support is contained in the rooted subtree $T_v$. In particular, notice that $\rist_G(v) \leq \St_G(|v|) \leq \St_G(v)$.

For $n\in \N$, the \emph{rigid stabiliser of the $n$\textsuperscript{th} level}, denoted by $\rist_G(n)$, is the subgroup of $G$ generated by the rigid stabilisers of all vertices on the $n$\textsuperscript{th} level. Notice that, as elements in the rigid stabilisers of two different vertices on the same level have disjoint support, they commute, and thus we have
\[\rist_G(n) \cong \prod_{v\in \level{n}}\rist_G(v).\]

\subsubsection{Branch groups and their properties}

We are now almost ready to define branch groups, which we do immediately after the following definition.

\begin{defn}
Let $T$ be a spherically homogeneous tree, as defined in Section \ref{subsubsec:RootedTrees}, and let $G\leq \Aut(T)$ be a group of automorphisms of $T$. Then, $G$ is said to act \emph{spherically transitively} if the action of $G$ on $\level{n}$ is transitive for every $n\in \N$.
\end{defn}

\begin{defn}\label{defn:BranchGroups}
Let $T$ be a spherically homogeneous rooted tree as defined in Section \ref{subsubsec:RootedTrees}, and let $G\leq \Aut(T)$ be a group of automorphisms of $T$. Then, $G$ is a \emph{branch group} if
\begin{enumerate}[label=(\roman*)]
\item $G$ acts spherically transitively on $T$,
\item \label{item:defBranchRistFiniteIndex} for every $n\in \N$, the subgroup $\rist_G(n)$ is of finite index in $G$.
\end{enumerate}
If, instead of \ref{item:defBranchRistFiniteIndex}, the group $G$ satisfies the weaker condition
\begin{enumerate}
\item[(ii)'] for every $n\in \N$, the subgroup $\rist_G(n)$ is non-trivial,
\end{enumerate}
then $G$ is said to be \emph{weakly branch}.
\end{defn}

In the rest of this section, we collect a few properties of branch groups that we will need later on. We begin by the following two lemmas, which are valid for any spherically transitive group.

\begin{lemma}\label{lemma:SectionsSphericallyTransitiveAreSphericallyTransitive}
Let $T$ be a spherically homogeneous rooted tree and let $G\leq \Aut(T)$ be a spherically transitive group of automorphisms of $T$. Then, for any $v\in T$, the group $G_v = \varphi_v(\St_G(v))$ is a spherically transitive group of automorphisms of $T_{|v|}$.
\end{lemma}
\begin{proof}
If $G$ acts spherically transitively on $T$, then $\St_G(v)$ must act spherically transitively on $T_v$, from which it follows that $G_v$ acts transitively on $T_{|v|}$.
\end{proof}

\begin{lemma}\label{lemma:TransitiveActionFiniteIndexBranch}
Let $T$ be a spherically homogeneous rooted tree and let $G\leq \Aut(T)$ be a spherically transitive group of automorphisms of $T$. Let $H\leq G$ be a subgroup of finite index of $G$. Then, there exists $N\in \N$ such that $\St_H(v)$ acts spherically transitively on $T_v$ for all $v\in \level{N}$.
\end{lemma}
\begin{proof}
For each $n\in \N$, let $l(n)$ denote the number of orbits of the action of $H$ on $\level{n}$. Notice that if two vertices on a given level are not in the same orbit, then their descendants must also be in different orbits. It follows that the sequence $(l(n))_{n\in \N}$ is non-decreasing. On the other hand, since $G$ acts transitively on $\level{n}$ for all $n\in \N$, the number of orbits in $\level{n}$ under the action of $H$ is bounded from above by the index of $H$ in $G$. Therefore, the sequence $(l(n))_{n\in \N}$ must eventually stabilise.

Let $N\in \N$ be such that $l(n)=l(N)$ for all $n\geq N$. For every $n\geq N$, one can partition $\level{n}$ into $l(N)$ sets by saying that two vertices belong to the same set if their prefixes of length $N$ are in the same $H$-orbit. These sets are obviously $H$-invariant, and since $l(n)=l(N)$, it follows that $H$ must act transitively on each of these sets. In particular, two vertices with the same prefix of length $N$ must belong to the same $H$-orbit. It follows that $\St_H(v)$  acts spherically transitively on $T_v$ for all $v\in \level{N}$.
\end{proof}

Note that one could formulate a similar version of the previous lemma in the more general context of groups with minimal actions by homeomorphisms on a topological space, but we will not require such generality here.

In the following lemma, we observe that for branch or weakly branch groups, the section map again yields a branch or weakly branch group.

\begin{lemma}\label{lemma:ProjectionOfRistIsInRist}
Let $G$ be a group of automorphisms of a spherically homogeneous rooted tree $T$ and let $v\in T$ be any vertex. Then, for any $n\geq |v|$, we have
\[\varphi_v(\rist_G(n))\leq \rist_{G_v}(n-|v|)\]
where $G_v=\varphi_v(\St_G(v))\leq \Aut(T_{|v|})$. In particular, if $G$ is a weakly branch group, then $G_v$ is also a weakly branch group, and if $G$ is a branch group, then so is $G_v$.
\end{lemma}
\begin{proof}
Let $v$ and $n$ be as above, and let $g\in \rist_G(n)$ be an arbitrary element. By definition, we can write
\[g=\prod_{w\in \level{n}}g_w\]
with $g_w\in \rist_G(w)$.

Since $n\geq |v|$, we have $g_w\in \St_G(v)$ for all $w\in \level{n}$. Therefore, we have
\[\varphi_v(g) = \prod_{w\in \level{n}}\varphi_v(g_w) = \prod_{vu\in \level{n}}\varphi_v(g_{vu}),\]
where the last equality follows from the fact that $\varphi_v(g_w)=1$ if $w\notin T_v$. For every $vu\in \level{n}$, we have $\varphi_v(g_{vu}) \in \rist_{G_v}(u)$. Indeed, let $w\in T_{|v|}$ be such that $w\notin T_u$. Then, $vw\notin T_{vu}$, which implies that $g_{vu}vw = vw$, since $g_{vu}\in \rist_G(vu)$, and thus $\varphi_v(g_{vu})w = w$, which shows that $\varphi_v(g_{vu})\in \rist_{G_v}(u)$. We deduce that $\varphi_v(\rist_G(n))\leq \rist_{G_v}(n-|v|)$. 

For the second part of the lemma, let us first notice that if $G$ acts spherically transitively on $T$, then $G_v$ acts spherically transitively on $T_{|v|}$ by Lemma \ref{lemma:SectionsSphericallyTransitiveAreSphericallyTransitive}. Thus, we only need to check condition $(ii)$ or $(ii)'$ from Definition \ref{defn:BranchGroups}.

If $G$ is a weakly branch group, then $\rist_G(w)$ is non-trivial for every $w\in T$, and it is not difficult to see that this implies that $\varphi_v(\rist_G(n))$ is non-trivial for any $n\geq |v|$. By the above, we conclude that $G_v$ is a weakly branch group. If $G$ is a branch group, then $\rist_G(n)$ is of finite index in $\St_G(v)$ for every $n\geq |v|$, from which we conclude that $\varphi_v(\rist_G(n))$ is of finite index in $G_v$, and thus it follows from the above that $G_v$ is also a branch group.
\end{proof}

The next Lemma is well-known and has appeared in various forms in several places. In the context of branch groups, it was first shown by Grigorchuk \cite{Grigorchuk00}. We present here a version that is more general than the one given by Grigorchuk, but the proof is essentially the same.
\begin{lemma}\label{lemma:SubnormalSubgroupsContainRist}
Let $G$ be a group of automorphisms of a rooted tree $T$, and let $H\leq G$ be a non-trivial subgroup of $G$. Let $v\in T$ be a vertex such that $Hv\ne \{v\}$ and let $R\leq \rist_G(v)$ be a subgroup of $\rist_G(v)$ normalising $H$. Then, $R'\leq H$, where $R'$ denotes the derived subgroup (also known as the commutator subgroup) of $R$. In particular, if $H$ is normal in $G$ and if $G$ acts spherically transitively on $T$, then $\rist_G'(|v|)\leq H$.
\end{lemma}
\begin{proof}
Let $h\in H$ be such that $hv\ne v$. For every $r\in R$, we have $hrh^{-1}\in \rist_G(hv)$, and this element must commute with every element of $R$, since $hv\ne v$. Thus, using the fact that $R$ normalises $H$, for every $r_1, r_2\in R$ we have
\begin{align*}
H\ni [[r_1,h],r_2] &= [r_1hr_1^{-1}h^{-1}, r_2] \\
&= [r_1,r_2].
\end{align*}
We conclude that $R'\leq H$.

If $H$ is a normal subgroup of $G$, we can take $R=\rist_G(v)$, so that $\rist_G'(v)\leq H$. If $G$ acts transitively on $\level{|v|}$, we conclude, using the fact that $H$ is normal, that $\rist_G'(w)\leq H$ for all $w\in \level{|v|}$, which implies that $\rist_G'(|v|)\leq H$.
\end{proof}

The previous lemma is more interesting if we know that the derived subgroup of a rigid stabiliser is non-trivial. This is always the case for weakly branch groups, as the following lemma shows.

\begin{lemma}\label{lemma:RistDerivedNonTrivial}
Let $G$ be a weakly branch group acting on a rooted tree $T$. For any $v\in T$ and any $k\in \N$, the subgroup $\rist_G^{(k)}(v)$ is non-trivial, where $\rist_G^{(0)}(v)=\rist_G(v)$ and $\rist_G^{(k+1)}(v) = [\rist_G^{(k)}(v), \rist_G^{(k)}(v)]$.
\end{lemma}
\begin{proof}
We will proceed by induction of $k$. For $k=0$, the result follows directly from the definition of a weakly branch group. Let us now suppose that for some $k\in \N$, we have $\rist_G^{(k)}(v)\ne 1$ for all $v\in T$, and let us show that the same statement holds for $k+1$.

Let us fix $v\in T$. Since $\rist_G^{(k)}(v)\ne 1$, there exist $w\in T$ and $r\in \rist_G^{(k)}(v)$ such that $rw\ne w$. Since $w$ is moved by an element of $\rist_G^{(k)}(v)\leq \rist_G(v)$, we must have $v\leq w$. Therefore, we have $\rist_G(w)\leq \rist_G(v)$, which implies that $\rist_G^{(k)}(w) \leq \rist_G^{(k)}(v)$.

By the induction hypothesis, we know that $\rist_G^{(k)}(w)\ne 1$, so there exists $h\in \rist_G^{(k)}(w)$ with $h\ne 1$. Since $\rist_G^{(k)}(w)\leq \rist_G^{(k)}(v)$, we have $[r,h]\in \rist_G^{(k+1)}(v)$. However, notice that $[r,h]=(rhr^{-1})h^{-1}$ with $rhr^{-1}\in \rist_G^{(k)}(rw)$ and $h^{-1}\in \rist_G^{(k)}(w)$ both non-trivial. Using the fact that $\rist_G^{(k)}(rw)\cap\rist_G^{(k)}(w) = \{1\}$, since $rw\ne w$, we conclude that $[r,h]\ne 1$. Thus, $\rist_G^{(k+1)}(v)\ne 1$, which concludes the proof.
\end{proof}

As an easy consequence of the previous lemmas, we obtain the following technical results, which will be useful later on.

\begin{lemma}\label{lemma:CanFindRistInMicroSuppSubgroup}
Let $G$ be a weakly branch group acting on a rooted tree $T$, let $H\leq G$ be a subgroup of finite index of $G$ and let $K\trianglelefteq H$ be a normal subgroup of $H$ such that $\rist_K(v)\ne 1$ for all $v\in T$. Then, there exists $n\in \N$ such that $\rist_G'(n) \leq K$.
\end{lemma}
\begin{proof}
Since $H$ is of finite index in $G$, it contains a non-trivial normal subgroup of finite index $N\trianglelefteq G$. By Lemma \ref{lemma:SubnormalSubgroupsContainRist}, there exists $m\in \N$ such that $\rist_G'(m)\leq N \leq H$. Using Lemma \ref{lemma:TransitiveActionFiniteIndexBranch} and the fact that $\rist_G'(m')\leq \rist_G'(m)$ for all $m'\geq m$, we may assume without loss of generality, replacing $m$ by a bigger number if necessary, that $\St_H(v)$ acts spherically transitively on $T_v$ for all $v\in \level{m}$.

For every $v\in \level{m}$, we have $\rist_K(v)\ne 1$. Therefore, there exists $u_v\in T_v$ such that $\rist_K(v)u_v\ne \{u_v\}$. Notice that if $\rist_K(v)u_v\ne \{u_v\}$, then $\rist_K(v)w\ne \{w\}$ for all $w\geq u_v$. Therefore, we may assume that there exists $k\in \N$ such that $|u_v|=k$ for all $v\in \level{m}$.

Since $\rist_G'(m)\leq H$, it normalises $K$, so by Lemma \ref{lemma:SubnormalSubgroupsContainRist}, for every $v\in \level{m}$, we have $\rist_G^{(2)}(u_v) \leq K$. By the spherical transitivity of the action of $H$ on each $T_v$, we conclude that $\rist_G^{(2)}(k)\leq K$. Since $\rist_G^{(2)}(k)$ is a normal subgroup of $G$, by applying Lemma \ref{lemma:SubnormalSubgroupsContainRist}, we obtain some $n\in \N$ such that $\rist_G'(n)\leq \rist_G^{(2)}(k)\leq K$.
\end{proof}

\begin{lemma}\label{lemma:CommutingSubgroupsDisjointSupport}
Let $G$ be a weakly branch group acting on a rooted tree $T$, let $H\leq G$ be a subgroup of finite index of $G$, and let $K_1,K_2\leq H$ be two subnormal subgroups of $H$. If $[K_1,K_2]=1$, then $K_1$ and $K_2$ have disjoint support.
\end{lemma}
\begin{proof}
Since $H$ is of finite index in $G$, it contains a non-trivial normal subgroup, and it thus follows from Lemma \ref{lemma:SubnormalSubgroupsContainRist} that there exists $N\in \N$ such that $\rist_G'(N)\leq H$.

Let us fix $i\in \{1,2\}$, and let $k_i\in \N$ be such that $K_i$ is a $k_i$-subnormal subgroup of $H$. Let $v\in T$ be a vertex in the support of $K_i$, meaning that $K_i v \ne \{v\}$. Then, by inductively applying Lemma \ref{lemma:SubnormalSubgroupsContainRist}, we obtain that $\rist_G^{(k_i+1)}(v)\leq K_i$.

Therefore, if there existed $v\in T$ such that $v$ were in the support of both $K_1$ and $K_2$, we would have $\rist_G^{(k_1+1)}(v)\leq K_1$ and $\rist_G^{(k_2+1)}(v)\leq K_2$, and thus $\rist_G^{(k+1)}\leq K_i$ for $i=1,2$, where $k=\max\{k_1,k_2\}$. Since $[K_1,K_2]=1$, this would imply that $\rist_G^{(k+2)}(v)=1$, which is impossible according to Lemma \ref{lemma:RistDerivedNonTrivial}.
\end{proof}

We conclude this section by observing that every proper quotient of a branch group has the property $BF$ defined in Definition \ref{defn:PropertyBF}.

\begin{prop}\label{prop:QuotientsBranchGroupsBF}
Every proper quotient of a branch group $G$ has property $BF$.
\end{prop}
\begin{proof}
Let $N\trianglelefteq G$ be a non-trivial normal subgroup of $G$. Let $H\leq G/N$ be a subgroup of finite bi-index. We need to show that $H$ has finite index in $G/N$.

It follows from Lemma \ref{lemma:SubnormalSubgroupsContainRist} that $G/N$ is a virtually abelian group. Therefore, there exists a subgroup $K\leq G/N$ of finite index in $G/N$ such that $K$ is abelian. According to Proposition \ref{prop:IntersectionFiniteBiIndex}, the subgroup $H\cap K$ is of finite bi-index in $K$. Since $K$ is abelian, a subgroup of finite bi-index must obviously be a subgroup of finite index. Thus, we have that $H\cap K$ is of finite index in $K$, which is itself of finite index in $G/N$. It follows that $H$ is of finite index in $G/N$.
\end{proof}

\subsection{Transitivity degree of branch groups}

Before moving on with the proof of our main result on quasi-$2$-transitive actions of branch groups, let us first mention the following result about the transitivity degree of branch groups due to Adrien Le Boudec, which, as mentioned in the introduction, served as a motivation for the current work on the more subtle problem of the quasi-transitivity degree. While this result will not be used in proof of the main theorem and can be deduced from it, we believe that it is interesting to show here how it can be proved directly with different arguments.

\begin{prop}[Le Boudec]\label{prop:TransitivityDegree}
The transitivity degree of a branch group is always $1$.
\end{prop}
\begin{proof}
The result follows from the two following general facts.
\begin{enumerate}[label=(\roman*)]
\item If $G$ is an infinite group acting faithfully and $2$-transitively on a set $X$, and if $H\leq G$ is a subgroup of finite index, then $H$ acts primitively on $X$.\label{item:FiniteIndex2TransitiveIsPrimitive}
\item A group acting faithfully and primitively on a set $X$ cannot contain three distinct non-trivial commuting normal subgroups.\label{item:PrimitiveNoCommutingSubgroups}
\end{enumerate}
Recall than an action is primitive if it is transitive and does not preserve any non-trivial partition, where a partition is said to be trivial if either every element is in a different block or if all elements are in the same block.

Indeed, if $G$ is a branch group, then there exists some $n\in \N$ such that $\level{n}$ contains at least three different vertices $v_1,v_2,v_3\in \level{n}$, and thus $\rist_G(n)$, which is a subgroup of finite index of $G$, contains three distinct non-trivial commuting normal subgroups, namely $\rist_G(v_1)$, $\rist_G(v_2)$ and $\rist_G(v_3)$. It then follows from \ref{item:FiniteIndex2TransitiveIsPrimitive} and \ref{item:PrimitiveNoCommutingSubgroups} that $G$ cannot act faithfully and $2$-transitively on a set and thus that its transitivity degree must be $1$.

To complete the proof, it remains only to show \ref{item:FiniteIndex2TransitiveIsPrimitive} and \ref{item:PrimitiveNoCommutingSubgroups}. Item \ref{item:FiniteIndex2TransitiveIsPrimitive} is Theorem 7.2D in \cite{DixonMortimer96}. For \ref{item:PrimitiveNoCommutingSubgroups}, let $G$ be a group acting faithfully and primitively on a set $X$, and let $N_1, N_2,N_3 \trianglelefteq G$ be three non-trivial commuting normal subgroup. For $i=1,2,3$, as $N_i$ is a normal subgroup, its orbits form a $G$-invariant partition of $X$, and thus from faithfulness and non-triviality, we conclude that $N_i$ acts transitively on $X$. Let us fix $x\in X$, and let $g_1\in N_1$ be any element. Then, by transitivity, there exists $g_2\in N_2$ such that $g_2^{-1}g_1x = x$. Using the fact that $N_3$ commutes with $N_1$ and $N_2$, we see that $g_3x = g_2^{-1}g_1g_3x$. As $N_3$ acts transitively on $X$, we see that $g_2^{-1}g_1y = y$ for all $y\in X$. By faithfulness, we have $g_2^{-1}g_1=1$, which implies that $g_1\in N_2$. As $g_1$ was arbitrary, we get $N_1\leq N_2$, and by symmetry, we have $N_1=N_2$, a contradiction to our assumption that $N_1\ne N_2$.
\end{proof}

\section{Quasi-prodense subgroups of branch groups}\label{sec:ProdenseSubgroups}

In this section, we define the notion of a quasi-prodense subgroup and we obtain a few useful results about quasi-prodense subgroups of branch groups.
These results are mostly generalisations of results on prodense subgroups of branch groups obtained by the author in \cite{Francoeur20a}. Unfortunately, although the main ideas are the same, the results in \cite{Francoeur20a} are not sufficiently strong for our purposes and cannot be used to deduce what we require. We therefore present here complete proofs of all results, following closely the outline of the proofs in \cite{Francoeur20a} but with a few key modifications necessary to deal with the added complications.

Let us begin with a definition.

\begin{defn}
A subgroup $H$ of a group $G$ is said to be \emph{quasi-prodense} if its \emph{pro-normal closure}
\[\overline{H}=\bigcap_{\substack{N\trianglelefteq G\\ N\ne 1}}HN\]
is of finite index in $G$.
\end{defn}

This notion is very similar to the notion of \emph{finite proindex} subgroup defined by de Cornulier in \cite{Cornulier06}. A subgroup $H$ of a group $G$ is said to be of finite proindex if its closure in the profinite topology is of finite index. Here, however, we consider a different kind of closure. In some groups, including branch and weakly branch groups, the pro-normal closure of a subgroup is indeed the closure of this subgroup in a topology, finer than the profinite topology, known as the \emph{pro-normal} topology (see for example \cite{GlasnerSoutoStorm10}). However, the pro-normal topology, unlike the profinite topology, cannot be defined for every group, although the definition of the pro-normal closure given above makes sense for all groups.

As the next lemma shows, in branch and weakly branch groups, sections of quasi-prodense subgroups are also quasi-prodense.

\begin{lemma}\label{lemma:ProjectionOfDenseIsDense}
Let $G$ be a weakly branch group acting on a rooted tree $T$. If $H\leq G$ is a quasi-prodense subgroup of $G$, then for every $v\in T$, the subgroup $H_v=\varphi_v(\St_H(v))$ is a quasi-prodense subgroup of $G_v=\varphi_v(\St_G(v))$.
\end{lemma}
\begin{proof}
Let us fix $v\in T$. We have
\[\overline{H_v} = \bigcap_{\substack{N\trianglelefteq G_v\\ N\ne 1}}H_vN \leq \bigcap_{n\in \N}H_v\rist'_{G_v}(n).\]
On the other hand, since $G_v$ is a weakly branch group by Lemma \ref{lemma:ProjectionOfRistIsInRist}, for every non-trivial normal subgroup $N\trianglelefteq G$, there exists some $n\in \N$ such that $\rist'_{G_v}(n)\leq N$ by Lemma \ref{lemma:SubnormalSubgroupsContainRist}. Therefore, we have
\[\overline{H_v} = \bigcap_{n\in \N}H_v\rist'_{G_v}(n).\]

Since $H$ is quasi-prodense by assumption, we have by definition that $\overline{H}$ is of finite index in $G$, which implies that $(\overline{H})_v=\varphi_v(\St_{\overline{H}}(v))$ is of finite index in $G_v$.
Let $h\in \St_{\overline{H}}(v)$ be an arbitrary element. Then, for any $n\in \N$, there exist $h_n\in H$ and $r_n\in \rist'_G(n)$ such that $h=h_nr_n$. For every $n\in \N$ greater that $|v|$, we have $r_n\in \rist'_G(n) \leq \St_G(v)$ and thus $h_n\in \St_H(v)$. Therefore, we have 
\[\varphi_v(h)=\varphi_v(h_n)\varphi_v(r_n)\in H_v\rist'_{G_v}(n-|v|)\]
by Lemma \ref{lemma:ProjectionOfRistIsInRist}.
This implies that $\varphi_v(h)\in \bigcap_{n\in \N}H_v\rist'_{G_v}(n)$ and thus that $(\overline{H})_v\leq \overline{H_v}$. As $(\overline{H})_v$ is of finite index in $G_v$, we conclude that $H_v$ is quasi-prodense in $G_v$.
\end{proof}

\begin{lemma}\label{lemma:RistNonTrivial}
Let $G$ be a weakly branch group acting on a rooted tree $T$ and let $H$ be a quasi-prodense subgroup of $G$. Suppose furthermore that for every vertex $v\in T$ different from the root, the subgroup $H_v=\varphi_v(\St_H(v))$ is of finite index in $G_v=\varphi_v(\St_G(v))$. Then, we must have $\rist_H(v)\ne 1$ for all $v\in T$.
\end{lemma}
\begin{proof}
We begin by fixing some notation that we will use throughout the proof. If $U\subset T$ is a set of vertices of $T$, we define the subgroup
\[\rist_H(U) = \left\{h\in H \hspace{0.3em}\Big\vert\hspace{0.3em} hw=w \quad \forall w\notin \bigcup_{u\in U} T_u\right\}\cap \bigcap_{u\in U}\St_H(u).\]
In particular, we have $\rist_H(\{v\}) = \rist_H(v)$ for all $v\in T$, and $\rist_H(\level{n}) = \St_H(n)$ for all $n\in \N$. Take notice that for convenience, we have given here a slightly different definition of $\rist_H(U)$ than the one given in \cite{Francoeur20a}.

Now, let us suppose for a contradiction that there exists a vertex $v\in T$ different from the root such that $\rist_H(v)=1$, and let us set $n=|v|$. Let $U\subseteq \level{n}$ be a subset of minimal cardinality such that $\varphi_v(\rist_H(U)) \ne 1$. Such as subset exists, since $\varphi_v(\rist_H(\level{n}))=\varphi_v(\St_H(n))$, and as $\varphi_v(\St_H(n))$ is of finite index in $\varphi_v(\St_H(v))=H_v$, which is itself of finite index in the weakly branch group $G_v$ by assumption, we know that $\varphi_v(\St_H(n))\ne 1$. Notice that we must have $v\in U$, otherwise any element of $\rist_H(U)$ would act trivially on $T_v$, which would imply that $\varphi_v(\rist_H(U))=1$. Since $\rist_H(v)=1$, it follows that $|U|\geq 2$. Let us pick some $u\in U\setminus \{v\}$.

Writing $K_1=\rist_H(U)$ and $K_2=\rist_H(\level{n}\setminus \{u\})$, we have
\[[K_1,K_2] \leq \rist_H(U\setminus\{u\}).\]
Indeed, if $w\in \level{n}$ is a vertex not belonging to $U\setminus \{u\} = U\cap (\level{n}\setminus \{u\})$, then $K_1T_w = K_2T_w = T_w$, since both $K_1$ and $K_2$ are in $\St_H(n)$, and one of $K_1$ or $K_2$ must act trivially on $T_w$, since $w$ is either not in $U$ or not in $\level{n}\setminus \{u\}$, from which it follows that $[K_1,K_2]$ must also act trivially on $T_w$.
By the minimality of the cardinality of $U$, we must have $\varphi_v(\rist_H(U\setminus\{u\}))=1$. It follows that 
\[[\varphi_v(K_1), \varphi_v(K_2)] = \varphi_v([K_1,K_2])\leq \varphi_v(\rist_H(U\setminus\{u\})) = 1,\]
and thus that $\varphi_v(K_1)$ and $\varphi_v(K_2)$ commute.

Both $K_1$ and $K_2$ are normal subgroups of $\St_H(n)$, which is itself a normal subgroup of finite index of $\St_H(v)$. Thus, $\varphi_v(K_1)$ and $\varphi_v(K_2)$ are two commuting subnormal subgroups of $H_v$, which is by assumption a finite index subgroup of the weakly branch group $G_v$. It follows from Lemma \ref{lemma:CommutingSubgroupsDisjointSupport} that $\varphi_v(K_1)$ and $\varphi_v(K_2)$ must have disjoint support. Since $\varphi_v(K_1)\ne 1$ by the definition of $K_1$, there must exist some $w\in T_{|v|}$ such that $\varphi_v(K_1)w \ne \{w\}$, and it follows that we must have $\varphi_v(K_2)w = \{w\}$. Furthermore, since every descendant of $w$ will also be moved by $\varphi_v(K_1)$, they must all be fixed by $\varphi_v(K_2)$. In other words, we have $\varphi_v(K_2)\leq \St_{H_v}(w)$ and $\varphi_w(\varphi_v(K_2))=1$. It follows that $K_2\leq \St_H(vw)$, with $\varphi_{vw}(K_2)=1$. 

Let us write $m=|w|$. The map
\begin{align*}
\alpha\colon \varphi_u(\St_H(n+m)) & \rightarrow \varphi_{vw}(\St_H(n+m)) \\
\varphi_u(g) & \mapsto \varphi_{vw}(g)
\end{align*}
is a well-defined homomorphism. Indeed, using the fact that both $\varphi_u$ and $\varphi_{vw}$ are homomorphisms when restricted $\St_H(n+m)$, if $g_1,g_2\in \St_H(n+m)$ are such that $\varphi_u(g_1)=\varphi_u(g_2)$, then $g_1g_2^{-1}\in K_2$, which implies that $\varphi_{vw}(g_1g_2^{-1})=1$ and thus $\varphi_{vw}(g_1)=\varphi_{vw}(g_2)$. We will show that the existence of such a homomorphism is in contradiction with the fact that $H$ is quasi-prodense.

Let $r\in \rist_{\overline{H}}(vw)$ be a non-trivial element. Such an element must exist, since $\overline{H}$ is a subgroup of finite index of the weakly branch group $G$. Notice however that $r\notin H$, since $\rist_H(vw) \leq \rist_H(v)=1$. As $r$ belongs to the rigid stabiliser of $vw$ and is non-trivial, we must have $\varphi_{vw}(r)\ne 1$, and therefore there must exist $m_1\in \N$ such that $\varphi_{vw}(r) \notin \St_{G_{vw}}(m_1)$.

Let us consider the subgroup $L=\alpha^{-1}(\varphi_{vw}(\St_{H}(n+m+m_1))) \geq \varphi_u(\St_{H}(n+m+m_1))$. As $\St_H(n+m+m_1)$ is of finite index in $\St_H(u)$, the subgroup $L$ is of finite index in $H_u$, which is itself of finite index $G_u$. Thus, by Lemma \ref{lemma:SubnormalSubgroupsContainRist},there exists some $m_2\in \N$ such that $\rist'_{G_u}(m_2)\leq L$.

Now, let $m_3=\max\{m_1,m_2-m\}$. By definition, we have $\overline{H}\leq H \rist'_G(n+m+m_3)$. This implies that $r=hg$ for some $h\in H$ and some $g\in \rist'_G(n+m+m_3)$. Since we have $r,g\in \St_G(n+m)$, we must have $h\in \St_H(n+m)$. Therefore, applying $\varphi_u$ to both side of this equality, we find $\varphi_u(r)=\varphi_u(h)\varphi_u(g)$. As $r\in \rist_H(vw)$, we must have $\varphi_u(r)=1$. Therefore, $\varphi_u(h) = \varphi_u(g^{-1})$. Since $g\in \rist'_G(n+m+m_3)$, we conclude, using Lemma \ref{lemma:ProjectionOfRistIsInRist}, that
\[\varphi_u(h)\in \varphi_u(\rist'_G(n+m+m_3))\leq \rist'_{G_u}(m+m_3)\leq \rist'_{G_u}(m_2) \leq L.\]

Using the fact that $\varphi_{vw}(h)=\alpha(\varphi_u(h))$, we see that 
\[\varphi_{vw}(h)\in \alpha(L) = \varphi_{vw}(\St_H(n+m+m_1)) \leq \St_{H_{vw}}(m_1).\]
Since we also have $g\in \rist'_G(n+m+m_3)\leq  \St_G(n+m+m_1)$, we also get $\varphi_{vw}(g)\in \St_{G_{vw}}(m_1)$.

Applying $\varphi_{vw}$ to both sides of the equality $r=hg$, we obtain $\varphi_{vw}(r)=\varphi_{vw}(h)\varphi_{vw}(g)$, and since $\varphi_{vw}(h),\varphi_{vw}(g)\in \St_{G_{vw}}(m_1)$, we conclude that $\varphi_{vw}(r)\in \St_{G_{vw}}(m_1)$. However, this is a contradiction, since we established above that $\varphi_{vw}(r) \notin \St_{G_{vw}}(m_1)$. We conclude that our initial assumption that $\rist_H(v)=1$ was false.
\end{proof}

\begin{thm}\label{thm:ProjectionsAreProper}
Let $G$ be a weakly branch group acting on a rooted tree $T$ and let $H<G$ be a quasi-prodense subgroup of infinite index. Then, for every $n\in \N$, there exists a vertex $v\in \level{n}$ such that $H_v=\varphi_v(\St_H(v))$ is a quasi-prodense subgroup of infinite index of $G_v=\varphi_v(\St_G(v))$.
\end{thm}
\begin{proof}
To prove the result, it suffices to show that there exists a vertex $v\in T$ such that $H_v$ is of infinite index in $G_v$. Indeed, if that is the case, then we may assume without loss of generality that $v\in \level{1}$, since $H_u$ being of infinite index in $G_u$ for some $u\in T$ implies that $H_w$ is of infinite index in $G_w$ for all $w\leq u$. It then follows from Lemma \ref{lemma:ProjectionOfDenseIsDense} that the result is true for $n=1$. By induction, we then conclude that the result is true for all $n\in \N$.

For the sake of contradiction, let us suppose that $H_v$ is of finite index in $G_v$ for all $v\in T$ except the root. Then, by Lemma \ref{lemma:RistNonTrivial}, we have $\rist_H(v)\ne 1$ for all $v\in T$. Thus, as $\rist_H(v)$ is a normal subgroup of $\St_H(v)$, we see that $\varphi_v(\rist_H(v))$ is a non-trivial normal subgroup of $H_v$, which is itself a subgroup of finite index of the weakly branch group $G_v$. Therefore, it follows from Lemma \ref{lemma:CanFindRistInMicroSuppSubgroup} that there exists some $m_v\in \N$ such that $\rist'_{G_v}(m_v)\leq \varphi_v(\rist_{H}(v))$, which implies, using Lemma \ref{lemma:ProjectionOfRistIsInRist}, that $\varphi_v(\rist'_G(m_v+|v|))\leq \varphi_v(\rist_H(v))$.

Let us set $m=\max\{m_v \mid v\in \level{1}\}$. Then, by the above argument, we have $\varphi_v(\rist'_G(m+1))\leq \varphi_v(\rist_H(v))$ for all $v\in \level{1}$. Since $\varphi_v(\rist'_G(m+1)) = \varphi_v(\rist'_G(m+1)\cap \rist_G(v))$, using the fact that $\varphi_v$ is injective on $\rist_G(v)$, we find that $\rist'_G(m+1)\cap \rist_G(v)\leq H$. As this is true for every $v\in \level{1}$, we conclude that $\rist'_G(m+1)\leq H$.

However, this contradicts the fact that $H$ is a quasi-prodense subgroup of infinite index. Indeed, we have $\overline{H}\leq H\rist'_G(m+1) =H$, with $\overline{H}$ of finite index and $H$ of infinite index, which is absurd. Thus, we conclude that there must exist some vertex $v\in T$ different from the root such that $H_v$ is of infinite index. This concludes the proof.
\end{proof}

Using this result, one can show the following proposition, which will be crucial in the next section.
\begin{prop}\label{prop:NbProjectionsInfiniteIndexUnbounded}
Let $G$ be a weakly branch group acting on a rooted tree $T$, and let $H<G$ be a quasi-prodense subgroup of infinite index. Then, for every $M\in \N$, there exists $n\in \N$ such that
\[|\{v\in\level{n} \mid [G_v:H_v]=\infty\}| > M,\]
where $H_v=\varphi_v(\St_H(v))$ and $G_v=\varphi_v(\St_G(v))$.
\end{prop}
\begin{proof}
By the definition of a quasi-prodense subgroup, the subgroup
\[\overline{H}=\bigcap_{\substack{N\trianglelefteq G\\ N\ne 1}}HN\]
must be of finite index in $G$. Therefore, by Lemma \ref{lemma:TransitiveActionFiniteIndexBranch}, there exists $m\in \N$ such that $\St_{\overline{H}}(v)$ acts spherically transitively on $T_v$ for every $v\in \level{m}$. It follows that for every $M\in \N$, one can find $n\in \N$ such that
\[|\overline{H}w|> M\]
for every $w\in \level{n}$.

Notice that for every $u\in T$, we must have $Hu=\overline{H}u$. Indeed, it is clear that $Hu\subseteq \overline{H}u$, and to show the reverse inclusion, one simply has to observe that $\overline{H}\leq H\St_G(|u|)$, so $\overline{H}u \subseteq H\St_G(|u|)u = Hu$.

Therefore, for every $M\in \N$, there exists $n\in \N$ such that $|Hw|> M$ for every $w\in \level{n}$. By Theorem \ref{thm:ProjectionsAreProper}, there exists $v\in \level{n}$ such that $H_v$ is a quasi-prodense subgroup of infinite index of $G_v$. This implies that $[G_u:H_u]=\infty$ for every $u\in Hv$, since in this case, $H_u$ and $G_u$ are conjugates of $H_v$ and $G_v$. As $|Hv|> M$, the result follows.
\end{proof}

\section{Quasi-2-transitive actions of branch groups}\label{sec:Proof}

In this section, we finally prove our main result, namely that a finitely generated branch group cannot have a quasi-2-transitive action on an infinite set. In order to do this, we will show that branch groups possess property $BF$, as defined in Definition \ref{defn:PropertyBF}. In other words, we will show that every subgroup of infinite index of a finitely generated branch group must have infinite bi-index. We will then be able to conclude using Proposition \ref{prop:Infinite2TransitiveFiniteBiIndex}.

Before we begin, we first show that it suffices to restrict our attention to quasi-prodense subgroups.

\begin{lemma}\label{lemma:LookOnlyAlmostProDense}
Let $G$ be a group whose proper quotients possess property $BF$. If $H<G$ is a subgroup of $G$ of finite bi-index, then $H$ is quasi-prodense in $G$.
\end{lemma}
\begin{proof}
Let $H<G$ be a subgroup of finite bi-index and let us consider its pro-normal closure
\[\overline{H}=\bigcap_{\substack{N\trianglelefteq G\\ N\ne 1}}HN.\]
We want to show that $\overline{H}$ is of finite index in $G$.

Notice first that for every non-trivial normal subgroup $1\neq N\trianglelefteq G$, the subgroup $HN$ must also be of finite bi-index. Consequently, $HN/N$ is a subgroup of finite bi-index in $G/N$. As every proper quotient of $G$ possesses property $BF$ by assumption, $HN/N$ must be of finite index in $G/N$, and thus $HN$ must be of finite index in $G$.

Furthermore, as $H$ is a subgroup of finite bi-index, it is almost maximal by Proposition \ref{prop:FiniteBiIndexAlmostMaximal}, meaning that it is properly contained in only finitely many subgroups. We conclude that $\overline{H}$ is an intersection of finitely many subgroups of finite index, and it thus of finite index.
\end{proof}

By Proposition \ref{prop:QuotientsBranchGroupsBF}, every proper quotient of a branch group has property $BF$. Thus, in light of Lemma \ref{lemma:LookOnlyAlmostProDense}, it suffices to show that a quasi-prodense subgroup of infinite index of a branch group must also have infinite bi-index.

Before we do this, we first need a technical result.

\begin{lemma}\label{lemma:InvariantBiCosets}
Let $G$ be a group acting on a spherically homogeneous rooted tree $T$, and let $H<G$ be a subgroup of $G$. For every $n\in \N$, let us define a map
\begin{align*}
\nu_n \colon \St_G(n) &\longrightarrow \N \\
g &\longmapsto |\{v\in \level{n} \mid \varphi_v(g)\in H_v\}|
\end{align*}
where $H_v = \varphi_v(\St_H(v))$. Then, $\nu_n$ is an invariant of bi-cosets, in the sense that for every $g_1,g_2\in \St_G(n)$ with $Hg_1H=Hg_2H$, we have $\nu_n(g_1)=\nu_n(g_2)$.
\end{lemma}
\begin{proof}
Let us fix $n\in \N$, and let $g_1,g_2\in \St_G(n)$ be such that $Hg_1H=Hg_2H$. Then, there exist $h_1,h_2 \in H$ such that $g_2=h_1g_1h_2$. Notice that since both $g_1$ and $g_2$ are in the stabiliser of the $n$\textsuperscript{th} level, for every $v\in \level{n}$ we must have $h_1h_2v=v$.

For every $v\in \level{n}$, we have
\[\varphi_v(g_2)=\varphi_{g_1h_2v}(h_1)\varphi_{h_2v}(g_1)\varphi_v(h_2)=\varphi_{h_2v}(h_1)\varphi_{h_2v}(g_1)\varphi_v(h_2),\]
where the last equality follows from the fact that $g\in \St_G(n)$. We want to show that $\varphi_v(g_2)$ is in $H_v$ if $\varphi_{h_2v}(g_1)$ is in $H_{h_2v}$. Suppose that $\varphi_{h_2v}(g_1)\in H_{h_2v}$. Then, by definition, there exists $h_3\in \St_H(h_2v)$ such that $\varphi_{h_2v}(h_3)=\varphi_{h_2v}(g_1)$. Thus, we have
\[\varphi_v(g_2)=\varphi_{h_2v}(h_1)\varphi_{h_2v}(h_3)\varphi_v(h_2) = \varphi_{h_3h_2v}(h_1)\varphi_{h_2v}(h_3)\varphi_v(h_2) = \varphi_v(h_1h_3h_2),\]
and since $h_1h_3h_2 \in \St_H(v)$ due to the fact that $h_3h_2v=h_2v$ and $h_1h_2v=v$, we conclude that $\varphi_v(g_2) \in  H_v$.

We have thus shown that for each $v\in \level{n}$, we have $\varphi_v(g_2)\in H_v$ if $\varphi_{h_2v}(g_1)\in H_{h_2v}$. Since $h_2$ is a bijection on $\level{n}$, we conclude that $\nu_n(g_2)\geq\nu_n(g_1)$. By symmetry, we must have $\nu_n(g_1) = \nu_n(g_2)$.
\end{proof}

We are now in position to prove that quasi-prodense subgroups of infinite index of branch groups are of infinite bi-index.

\begin{lemma}\label{lemma:BranchFiniteProindexFiniteBiIndex}
Let $G$ be a branch group and let $H<G$ be a quasi-prodense subgroup of infinite index. Then, $H$ is of infinite bi-index.
\end{lemma}
\begin{proof}
We will show that for any $M\in \N$, the number of bi-cosets of $H$ must be larger than $M$.

Let us fix an arbitrary $M\in \N$. By Proposition \ref{prop:NbProjectionsInfiniteIndexUnbounded}, there exists $n\in \N$ such that $ |\mathcal{V}_n|> M$, where
\[\mathcal{V}_n=\{v\in\level{n} \mid [G_v:H_v]=\infty\}\]
with $G_v=\varphi_v(\St_G(v))$ and $H_v = \varphi_v(\St_H(v))$.

Since $G$ is a branch group, $\rist_G(n)$ is of finite index in $G$. It follows that $\varphi_v(\rist_G(n))$ is of finite index in $G_v$ for all $v\in \level{n}$, and thus in particular for all $v\in \mathcal{V}_n$. It then follows that for every $v\in \mathcal{V}_n$, one can find an element $g_v\in \rist_G(v)$ such that $\varphi_v(g_v) \notin H_v$, using the fact that the index of $H_v$ in $G_v$ is infinite. Notice that since $g_v\in \rist_G(v)$, we have $\varphi_w(g_v)=1\in H_w$ for all $w\in \level{n}$ with $w\ne v$.

Let us write $\mathcal{V}_n=\{v_1, v_2, \dots, v_{|\mathcal{V}_n|}\}$. For every $i\in \{1,2,\dots, |\mathcal{V}_n|\}$, we define an element $g_i\in \rist_G(n)$ as follows:
\[g_i = \prod_{k=1}^{i}g_{v_i}.\]
By construction, for $w\in \level{n}$, we have $\varphi_w(g_i)\in H_w$ if and only if $w\notin \{v_1, \dots, v_i\}$. Therefore, for every $i\in \{1,2, \dots, |\mathcal{V}_n|\}$, we have
\[\nu_n(g_i) = |\level{n}|-i,\]
where $\nu_n(g_i)$ is the map defined in Lemma \ref{lemma:InvariantBiCosets}. As $i$ ranges from $1$ to $|\mathcal{V}_n|$, we conclude that the map $\nu_n$ takes at least $|\mathcal{V}_n|>M$ different values. Since two elements in the same bi-coset must have the same image under $\nu_n$ according to Lemma \ref{lemma:InvariantBiCosets}, there must exist more than $M$ different bi-cosets. This concludes the proof.
\end{proof}

Thanks to the previous result, we are finally ready to show that branch groups have property $BF$.

\begin{thm}\label{thm:BranchGroupsInBF}
Let $G$ be a branch group and let $H<G$ be a subgroup of finite bi-index. Then, $H$ is of finite index in $G$. In other words, branch groups have property $BF$.
\end{thm}
\begin{proof}
If $H<G$ is a subgroup of finite bi-index, then it is a quasi-prodense subgroup of $G$ by Lemma \ref{lemma:LookOnlyAlmostProDense} and Proposition \ref{prop:QuotientsBranchGroupsBF}. Therefore, by Lemma \ref{lemma:BranchFiniteProindexFiniteBiIndex}, $H$ must be of finite index in $G$.
\end{proof}

As a consequence of the previous result, we obtain that branch groups do not admit quasi-2-transitive actions on infinite sets.

\begin{MainThm}
A branch group cannot act quasi-$2$-transitively on an infinite set.
\end{MainThm}
\begin{proof}
Let $G$ be a branch group. By Proposition \ref{prop:Infinite2TransitiveFiniteBiIndex}, to admit a quasi-$2$-transitive action of an infinite set, $G$ must contain a subgroup $H$ of infinite index but of finite bi-index, which is impossible according to Theorem \ref{thm:BranchGroupsInBF}.
\end{proof}

\bibliographystyle{plain}
\bibliography{biblio}

\end{document}